\pdfoutput=1
\documentclass[11pt,a4paper]{amsart}

\usepackage[utf8]{inputenc}

\usepackage[left=1in,right=1in,top=1in,bottom=1in]{geometry}

\usepackage{amssymb}
\usepackage{amsmath}
\usepackage{amsthm}
\usepackage{thmtools}
\usepackage{mleftright}
\allowdisplaybreaks

\usepackage{stmaryrd}

\usepackage{tikz}
\usepackage{tikz-cd}

\usepackage{enumitem}

\newlist{enumalpha}{enumerate}{1}
\setlist[enumalpha, 1]{
label=(\alph*)
}

\usepackage{mathtools}

\usepackage[hypertexnames=false,hidelinks]{hyperref}

\usepackage[english,noabbrev,sort]{cleveref}

\usepackage{my-math-definitions} 

\newtheorem{theorem}{Theorem}[section]
\newtheorem{lemma}[theorem]{Lemma}
\newtheorem{corollary}[theorem]{Corollary}
\theoremstyle{definition}
\newtheorem{definition}[theorem]{Definition}
\theoremstyle{remark}
\newtheorem{example}[theorem]{Example}
\newtheorem{remark}[theorem]{Remark}

\numberwithin{equation}{section}

\DeclareMathOperator{\jump}{exjump}

\DeclareMathOperator{\asc}{asc}
\DeclareMathOperator{\triv}{triv}
\newcommand{\Fasc}{F^{\asc}}
\newcommand{\J}{\mathbb J}

\author{Fabian Gundlach}
\address{Universität Paderborn, Fakultät EIM, Institut für Mathematik, Warburger Str.~100, 33098 Paderborn, Germany.}
\email{fabian.gundlach@uni-paderborn.de}
\title{Multivariate counting of wild abelian extensions}

\begin{document}

\begin{abstract}
Let $K$ be a global function field of characteristic~$p$ and let $G$ be a finite abelian group of exponent $p^e$.
We show that the multivariate generating function counting sub-$G$-extensions of $K$ with respect to $e$ specific height functions (encoding successive drops in the exponents of the higher ramification groups) is rational.
\end{abstract}

\maketitle

\section{Introduction}

Let $G$ be a finite abelian group and let $K$ be a global function field of characteristic $p$, i.e., the function field of a smooth projective geometrically irreducible curve $C_K$ over some finite field $\F_q$. Denote by $\Hom(\Gamma_K, G)$ the set of (not necessarily surjective) continuous group homomorphisms from the absolute Galois group of $K$ to the group $G$. We will call the elements of this set \emph{sub-$G$-extensions} of $K$.

Sub-$G$-extensions have been counted by various \emph{inertial invariants} / \emph{inertial height functions}, such as by the degree of the conductor divisor (see \cite{wood-probabilities-of-local-behaviors} for abelian groups of order coprime to~$p$, \cite{lagemann-artin-schreier-witt} for abelian $p$-groups, and \cite{he-equidistribution} for arbitrary abelian groups), discriminant divisor (see \cite{wright-counting-abelian-extensions} for abelian groups of order coprime to $p$ and \cite{potthast-elementary-abelian} for elementary abelian $p$-groups), and other invariants (see \cite{wood-probabilities-of-local-behaviors}, \cite{tavernier} for abelian groups of order coprime to~$p$, and \cite{gundlach-abelian-extensions-by-artin-schreier} for abelian $p$-groups). 
If $p$ is coprime to the order of $G$, then any local function field of characteristic $p$ has only boundedly many sub-$G$-extensions, and the space of reasonable \emph{inertial height functions} is finite-dimensional.
On the other hand, if $p$ divides the order of $G$, then any local function field of characteristic~$p$ has infinitely many sub-$G$-extensions, and the space of inertial height functions is infinite-dimensional. One may wonder for which inertial height functions the resulting counting problem has the cleanest answers.

We now assume that $G$ has exponent $p^e$. In \cite{gundlach-abelian-extensions-by-artin-schreier}, it was shown that if the function field $K$ is rational, then for one particular choice of inertial height function, namely the \emph{Artin--Schreier conductor}, the resulting generating function has the unusual property of being a rational function. In particular, this gives an exact (not just asymptotic) solution to the global counting problem.

In this article, we improve on this in two directions: (a) We remove the assumption that $K$ is rational. (b) In the spirit of \cite[page~164]{ellenberg-venkatesh-counting-galois-extensions} and \cite{gundlach-multiple-invariants}, and generalizing \cite[Remark~4.6]{gundlach-abelian-extensions-by-artin-schreier}, we count simultaneously by $e$ carefully chosen inertial height functions and show that the resulting multivariate generating function is again rational.

To define the height functions, we first associate to any $\varphi\in\Hom(\Gamma_K,G)$ the following divisors on the curve $C_K$:
\[
	\jump_i(\varphi)
	:= \sum_{P\textnormal{ place of }K}
	\inf\{v\in\R_{\geq0} : \varphi(\Gamma_{K_P}^v) \subseteq G[p^i]\} \cdot P
	\ \in\ \Div(C_K)
	\qquad\textnormal{for }i=0,1,\dots,
\]
where $\Gamma_{K_P}^v \trianglelefteq \Gamma_{K_P} \subseteq \Gamma_K$ denotes the $v$-th higher ramification subgroup at~$P$ (in upper numbering) and $G[p^i]$ denotes the $p^i$-torsion subgroup of $G$. Note that the divisors $\jump_i(\varphi)$ encode the exponents of the ramification subgroups since $\varphi(\Gamma_{K_P}^v)\subseteq G[p^i]$ is equivalent to $\varphi(\Gamma_{K_P}^v)$ having exponent dividing~$p^i$.

These divisors satisfy $\jump_i(\varphi)=0$ for all $i\geq e$ and $\jump_i(\varphi) \geq p\cdot\jump_{i+1}(\varphi)$ for all $i\geq0$. (See \Cref{lem:jump-restrictions} for a proof of this well-known fact.)
It is thus natural to perform the invertible change of variables
\[
	(\jump_0,\dots,\jump_{e-1},\underbrace{\jump_e}_0,\dots)\mapsto({\jump_0} - p\jump_1,\dots,{\jump_{e-1}}-p\jump_e),
\]
which gives rise to $e$ effective divisors $\jump_i(\varphi)-p\cdot\jump_{i+1}(\varphi)$ on the curve $C_K$.
We will count sub-$G$-extensions $\varphi\in\Hom(\Gamma_K,G)$ simultaneously by the $e$ inertial height functions
\[
	\varphi \mapsto \deg(\jump_i(\varphi)-p\cdot\jump_{i+1}(\varphi))
	\qquad\textnormal{for $i=0,\dots,e-1$}.
\]
Our main result is as follows:

\begin{theorem}[see \Cref{thm:main} for a more precise statement]\label{thm:intro-main}
	For any finite abelian group~$G$ of exponent $p^e$ and any global field~$K$ of characteristic~$p$, the multivariate generating function
	\[
		F_K(X_0,\dots,X_{e-1})
		:= \frac1{|G|}\sum_{\substack{\varphi\in\Hom(\Gamma_K, G)}}
			\prod_{i=0}^{e-1} X_i^{\deg(\jump_i(\varphi)-p\cdot\jump_{i+1}(\varphi))}
		\in \Z\llbracket X_0,\dots,X_{e-1}\rrbracket
	\]
	is rational.
\end{theorem}

The explicit description of the rational function $F_K(X_0,\dots,X_{e-1})$ in \Cref{thm:main} can be used (for given $K$ and $G$) to obtain exact formulas for the number of sub-$G$-extensions $\varphi$ with heights $\deg(\jump_i(\varphi)-p\cdot\jump_{i+1}(\varphi)) = n_i$. For $n_0,\dots,n_{e-1}\to\infty$, we give general asymptotics in \Cref{thm:Fasymp}.

Focusing on the divisor $\asc(\varphi) := \jump_0(\varphi)$, which was called the \emph{Artin--Schreier conductor} in \cite{gundlach-abelian-extensions-by-artin-schreier} and the \emph{last jump divisor} in \cite{gundlach-seguin-two-step-nilpotent,gundlach-seguin-local-global}, we obtain the following generalization of \cite[Corollary~4.5(b)]{gundlach-abelian-extensions-by-artin-schreier}:

\begin{theorem}[see \Cref{thm:main-asc}]\label{thm:intro-main-asc}
	For any nontrivial finite abelian $p$-group~$G$ and any global field~$K$ of characteristic~$p$, letting $a := 1 + \dim_{\F_p}(G[p])$:
	\begin{enumalpha}
	\item
		The single-variable generating function
		\[
			\Fasc_K(X)
			:= \frac1{|G|}\sum_{\substack{\varphi\in\Hom(\Gamma_K, G)}} X^{\deg(\asc(\varphi))}
			= F_K(X,X^p,\dots,X^{p^{e-1}})
			\in \Z\llbracket X\rrbracket
		\]
		is rational.
	\item
		Its only innermost pole is a simple pole at $X = q^{-a}$ with negative residue.
	\item
		We have the following asymptotic statement for some constant $C > 0$:
		\[
			|\{\varphi\in\Hom(\Gamma_K, G) : \deg(\asc(\varphi)) = n\}|
			\sim C q^{a n}
			\qquad\textnormal{ for }n\rightarrow\infty.
		\]
	\end{enumalpha}
\end{theorem}

\subsection*{Notation}

We denote the set of places of $K$ by $M_K$, its genus by $g_K$, its Picard group by $\Pic_K$, and the degree zero part of its Picard group by $\Pic^0_K$. For any place $P$, we let $K_P$ be the completion, $\O_P = \O_{K_P}$ its ring of integers, $v_P$ the $P$-adic valuation, $U_P^k = 1 + \pi_P^k\O_P$ the $k$-unit group (for $k\geq1$), $\pi_P$ any uniformizer, and $\kappa_P$ the residue field, and $Q_P$ the size of the residue field. For any continuous group homomorphism $\varphi:\Gamma_K\to G$ or $\varphi:\Gamma_{K_P}\to G$ and any $i\geq0$, we let
\[
	\jump_{P,i}(\varphi)
	:= \inf\{v\in\R_{\geq0} : \varphi(\Gamma_{K_P}^v) \subseteq G[p^i]\}
	= \min\{k\in\Z_{\geq0} : \varphi(\Gamma_{K_P}^{k+1}) \subseteq G[p^i]\}
	\in \Z_{\geq0}.
\]
(The equality and integrality here follow from the Hasse--Arf theorem, which states that in an abelian field extension, the jumps in the higher ramification filtration in upper numbering can only occur right \emph{after} integers.)
Note that if $\varphi$ is unramified at $P$, then $\jump_{P,i}(\varphi)=0$.

\subsection*{Strategy of proof}

Following for example \cite{wright-counting-abelian-extensions,lagemann-artin-schreier-witt,gundlach-abelian-extensions-by-artin-schreier}, we use class field theory to describe sub-$G$-extensions of $K$. For rational function fields, it provides a $|G|$-to-$1$ map from $\Hom(\Gamma_K, G)$ to $\bigoplus_{P\in M_K} \Hom(\O_P^\times, G)$. For non-rational function fields, this local-global principle can fail. The classical method to deal with this failure involves choosing a finite set $S$ of places generating the Picard group. We use a more direct approach, constructing an $n$-to-$1$ map from $\Hom(\Gamma_K, G)$ to the kernel of a boundary map $\delta : \bigoplus_{P\in M_K}\Hom(\O_P^\times,G)\to\Ext^1_\Z(\Pic^0_K,G)$, with $n = |G|\cdot|{\Ext^1_\Z(\Pic^0_K,G)}|$.

To count elements $\varphi$ of $\Hom(\Gamma_K, G)$ with prescribed values for $\jump_{P,i}(\varphi)$, we rely on a description of the filtration $\O_P^\times \supset U_P^1 \supset U_P^2 \supset \dots$, together with information about how this filtration interacts with the boundary map $\delta$. We obtain this information by looking at $p$-Selmer groups, making use of crucial ideas of Lagemann \cite{lagemann-artin-schreier}, but carefully working around a small mistake in \cite{lagemann-artin-schreier-witt} for groups of exponent larger than~$p$ (see \Cref{rmk:lagemann-higher-exponents}).

\subsection*{Acknowledgements}

This work was supported by the Deutsche Forschungsgemeinschaft (DFG, German Research Foundation) --- Project-ID 491392403 --- TRR 358 (project A4).

The author is grateful to Jordan Ellenberg and Kiran Kedlaya for asking whether \cite{gundlach-abelian-extensions-by-artin-schreier} generalizes to non-rational function fields, to Jürgen Klüners, Nicolas Potthast, and Bé\-ran\-ger Seguin for helpful discussions and comments on an earlier draft, and to Xujia Chen for instructing GPT to provide comments on the article.

\section{The Selmer group}\label{sn:selmer}

We first briefly recall some basic facts about the local unit groups $\O_P^\times$. The exact sequence
\[
	1 \rightarrow U_P^1 \rightarrow \O_P^\times \rightarrow \kappa_P^\times \rightarrow 1,
\]
splits by Hensel's lemma, so we obtain a natural isomorphism $\O_P^\times \simeq \kappa_P^\times \times U_P^1$.

\begin{lemma}[{\cite[Proposition~5.2b]{potthast-elementary-abelian}}]\label{rmk:units-mod-p-fund-dom}
	Every coset in $\O_P^\times/\O_P^{\times p}$ has exactly one representative of the form $1 + \sum_{i\geq1:\ p\nmid i} a_i\pi_P^i$ with $a_i\in\kappa_P$.
\end{lemma}
\begin{proof}
	 First, using that the order of $\kappa_P^\times$ is coprime to $p$, we can divide by a $p$-th power to make the constant coefficient~$1$. Then, divide by $p$-th powers to eliminate all $p$-power terms. (See for example \cite[Proposition~5.2b]{potthast-elementary-abelian} for more details.)
\end{proof}

To understand obstructions to the local-global principle for sub-$G$-extensions of $K$, we will use the \emph{$p$-Selmer group}
\[
	\Sel_K := \{ x \in K^\times : (x) = p D\textnormal{ for some divisor }D \} / K^{\times p},
\]
which is a (multiplicative) $\F_p$-vector space isomorphic to the $p$-torsion subgroup of $\Pic_K$ via the isomorphism
\[
	\Sel_K \stackrel\sim\to \Pic_K[p],\qquad
	[x] \mapsto [D]\quad \textnormal{if }(x)=pD.
\]
(The injectivity of this map relies on the fact that every element of the unit group $\F_q^\times$ is a $p$-th power.)

For any element $[y]$ of the $p$-Selmer group and any place $P$, we have $y \in K_P^{\times p}\cdot U_P^1$ since $y/\pi_P^{v_P(y)} \in \O_P^\times \simeq \kappa_P^\times\times U_P^1$ with $p \mid v_P(y)$ and since every element of $\kappa_P^\times$ is a $p$-th power. We define
\[
	\omega_P([y]) := \sup\{k\geq1 \mid y \in K_P^{\times p} \cdot U_P^{k}\} \in \Z_{\geq1}\cup\{\infty\}.
\]
\begin{remark}\label{rmk:omega-alt}
	If we write $[y] = [1+\sum_{i\geq1:\ p\nmid i} a_i\pi_P^i]$ using \Cref{rmk:units-mod-p-fund-dom}, then
	\[
		\omega_P([y]) = \inf\{i\geq1: p\nmid i,\ a_i\neq0\}.
	\]
\end{remark}

\begin{lemma}\label{lem:selmer}
	For every nontrivial element $[y]$ of the $p$-Selmer group, we have
	\[
		\sum_{P\in M_K} (\omega_P([y]) - 1) \cdot \deg(P) = 2g_K - 2,
	\]
	where $g_K$ is the genus of $K$. In particular, $\omega_P([y]) = 1$ for all but finitely many $P$.
\end{lemma}
\begin{proof}
	This is essentially \cite[Lemma~4.4]{lagemann-artin-schreier} and \cite[Proposition~6.14]{potthast-elementary-abelian}.
	Let $\Omega_K \subseteq \Omega_{K_P}$ be the groups of differentials on $K$ and $K_P$, respectively.
	The logarithmic derivatives
	\[
		l:K^\times/K^{\times p} \hookrightarrow \Omega_K
		\textnormal{ and }
		l_P:K_P^\times/K_P^{\times p} \hookrightarrow \Omega_{K_P},
		\qquad
		x\mapsto \frac{dx}{x}
	\]
	are injective group homomorphisms. (See for example \cite[Theorem~26.5]{matsumura-commutative-ring-theory} for injectivity.)
	
	As $y\notin K^{\times p}$, we have $l(y)\neq0$. The canonical divisor $(l(y))$ has degree $2g_K-2$, so we have $\sum_P v_P(l_P(y)) \cdot \deg(P) = 2g_K - 2$. It therefore suffices to show $\omega_P([y])-1 = v_P(l_P(y))$. Indeed, writing $[y]=[1+\sum_{i\geq1:\ p\nmid i}a_i\pi_P^i]$, \Cref{rmk:omega-alt} implies
	$\omega_P([y]) = \inf\{i\geq1: p\nmid i,\ a_i\neq0\}$, while on the other hand we have
	\[
		v_P(l_P(y))
		= v_P\Bigl(\sum_{\substack{i\geq1:\\p\nmid i}}ia_i\pi_P^{i-1} d\pi_P\Bigr)
		= \inf\{i\geq1:\ p\nmid i,\ a_i\neq0\} - 1.
	\qedhere
	\]
\end{proof}

\begin{corollary}\label{cor:selmer-embedding}
	\begin{enumalpha}
	\item\label{item:selmer-embedding-a}
		For every $P\in M_K$, there is a number $w_P\geq1$ such that the group homomorphism
		\[
			\Sel_K \rightarrow \O_P^\times / (\O_P^{\times p}\cdot U_P^{w_P+1}),\qquad
			[y] \mapsto [y/\pi_P^{v_P(y)}],
		\]
		which does not depend on the choice of uniformizer $\pi_P$, is injective.
	\item
		For all but finitely many $P$, we can take $w_P = 1$.
	\item
		If $g_K\leq1$, then we can take $w_P = 1$ for all $P$.
	\end{enumalpha}
\end{corollary}
\begin{proof}
	The claim follows from the lemma by taking $w_P$ to be the maximum value of $\omega_P([y])$ over the (finitely many) nontrivial elements $[y]$ of the Selmer group. (If the Selmer group is trivial, we can take $w_P=1$.)
\end{proof}

\begin{remark}
	\Cref{lem:selmer} and \Cref{cor:selmer-embedding} are special to characteristic $p$: If the characteristic of $K$ were different from $p$, then there would be infinitely many $P\in M_K$ for which the ($p$-Selmer) group homomorphism $\Sel_K \to \O_P^\times/\O_P^{\times p}$, $[y]\mapsto[y/\pi_P^{v_p(y)}]$, is trivial. Indeed, letting $[y_1],\dots,[y_t]$ be the elements of $\Sel_K$, the Chebotarev density theorem would imply that infinitely many places $P\in M_K$ split completely in the \emph{separable} field extension $K(\sqrt[p]{y_1},\dots,\sqrt[p]{y_t})$, which would mean that $y_i \in K_P^{\times p}$ and thus $y_i/\pi_P^{v_P(y_i)} \in \O_P^{\times p}$ for all $i$.
\end{remark}

\section{Class field theory}\label{sn:class-field-theory}

By class field theory, $\Gal(K^\ab | K)$ is isomorphic to the profinite completion of the idele class group $\A_K^\times / K^\times$, so elements of $\Hom(\Gamma_K, G) = \Hom(\Gal(K^\ab|K),G)$ are in bijection with continuous group homomorphisms (in the following just called \emph{maps}) $\varphi : \A_K^\times / K^\times \to G$. We henceforth identify the elements of $\Hom(\Gamma_K,G)$ with the maps $\A_K^\times/K^\times\to G$. Since the $k$-th higher ramification group for a place $P$ in $\Gal(K^\ab|K)$ corresponds to the image of $U_P^k$ in $\A_K^\times/K^\times$, we can compute $\jump_{P,i}(\varphi)$ as follows:
\begin{equation}\label{eq:lastjump-vs-kunits}
	\jump_{P,i}(\varphi)
	= \min\{k\in\Z_{\geq0} \mid \varphi(U_P^{k+1}) \subseteq G[p^i]\}.
\end{equation}

To understand the idele class group $\A_K^\times/K^\times$, we use the exact sequence
\[
	1
	\to \Big(\prod_{P\in M_K}\O_P^\times\Big)/\F_q^\times
	\to \A_K^\times/K^\times
	\to \Pic_K
	\to 1,
\]
in which the map $\A_K^\times/K^\times\to\Pic_K$ is given by $(x_P)_{P\in M_K} \mapsto \sum_{P\in M_K} v_P(x_P) \cdot P$.

Since $G$ is a finite $p$-group and $\F_q^\times$ has size coprime to $p$, we have a natural isomorphism $\Hom((\prod_P \O_P^\times)/\F_q^\times, G) \simeq \Hom(\prod_P \O_P^\times, G) \simeq \bigoplus_P \Hom(\O_P^\times, G)$.
We obtain the long exact sequence
\[
	1
	\to \Hom(\Pic_K, G)
	\to \Hom(\A_K^\times/K^\times, G)
	\to \bigoplus_{P\in M_K}\Hom(\O_P^\times, G)
	\stackrel\delta\to \Ext^1_\Z(\Pic_K, G).
\]
Since we have a (non-canonical) isomorphism $\Pic_K \simeq \Z \times \Pic^0_K$ and we have $\Hom(\Z,G)\simeq G$ and $\Ext^1_\Z(\Z,G) = 1$, this can be written as
\begin{equation}\label{eq:ext-exact-sequence}
	1
	\to G \times \Hom(\Pic^0_K, G)
	\to \Hom(\A_K^\times/K^\times, G)
	\to \bigoplus_{P\in M_K}\Hom(\O_P^\times, G)
	\stackrel\delta\to \Ext^1_\Z(\Pic^0_K, G).
\end{equation}

The group $\Ext^1_\Z(\Pic_K^0, G)$ and the boundary map $\delta$ can be described concretely as follows:
\begin{remark}\label{rmk:ext-made-concrete}
	Let $\Pic^0_K[p^\infty] \simeq \prod_{i=1}^t \Z/p^{n_i}\Z$, with a generator of the $i$-th factor given by the divisor class $[D_i]$ of order $p^{n_i}$. Then, $\Ext^1_\Z(\Pic^0_K, G) \simeq \prod_{i=1}^t G/p^{n_i} G$. In particular, $\Ext^1_\Z(\Pic^0_K,G)$ depends only on the $p$-part of the Picard group $\Pic_K$.
	
	If we choose $y_i\in K^\times$ so that $(y_i) = p^{n_i} D_i$, then the restriction of $\delta$ to $\Hom(\O_P^\times, G)$ corresponds to the map
	\[
		\delta : \Hom(\O_P^\times, G) \to \prod_{i=1}^t G/p^{n_i}G,
		\qquad
		f \mapsto \left(f\big(y_i / \pi_P^{v_P(y_i)}\big)\right)_{i=1,\dots,t}.
	\]
	(This map is independent of the choice of uniformizer $\pi_P$ as $p^{n_i} \mid v_P(y_i)$.)
	One can of course directly verify the exact sequence (\ref{eq:ext-exact-sequence}) using this explicit description of $\Ext^1_\Z(\Pic^0_K, G)$ and the map $\delta$.
\end{remark}

\begin{remark}\label{lem:linear-independence}
	Let $y_1,\dots,y_t$ be as in the previous remark. Their images $[y_1],\dots,[y_t]$ in $\Sel_K$ are $\F_p$-linearly independent since the corresponding divisor classes $[p^{n_1-1}D_1],\dots,[p^{n_t-1}D_t]$ in $\Pic_K^0[p]$ are $\F_p$-linearly independent.
\end{remark}

\section{The local unit groups}

\subsection{Constraints on the exponent jumps}

We first show the following well-known constraints on the values $\jump_{P,i}(\varphi)$:

\begin{lemma}\label{lem:jump-restrictions}
	Every $\varphi\in\Hom(\Gamma_{K_P},G)$ satisfies:
	\begin{enumalpha}
	\item
		$\jump_{P,i}(\varphi)=0$ for all $i\geq e$.
	\item
		$\jump_{P,i}(\varphi) \geq p\jump_{P,i+1}(\varphi)$ for all $i\geq0$.
	\end{enumalpha}
\end{lemma}
\begin{proof}
	Recall the description of $\jump_{P,i}(\varphi)$ in \Cref{eq:lastjump-vs-kunits}.
	\begin{enumalpha}
	\item
		This is clear since $\varphi(U_P^{k+1}) \subseteq G = G[p^i]$ for all $k\geq0$.
	\item
		Let $k := \jump_{P,i+1}(\varphi)$. The case $k=0$ being trivial, we assume $k\geq1$. By \Cref{eq:lastjump-vs-kunits}, $\varphi(U_P^k) \nsubseteq G[p^{i+1}]$, so $\varphi(U_P^k)^p \nsubseteq G[p^i]$. Since we are in characteristic $p$, we have $(U_P^k)^p \subseteq U_P^{pk}$ as $(1+\pi^k x)^p = 1 + \pi^{pk} x^p$, so this implies $\varphi(U_P^{pk}) \nsubseteq G[p^i]$. Thus, $\jump_{P,i}(\varphi) \geq pk = p\jump_{P,i+1}(\varphi)$.
	\qedhere
	\end{enumalpha}
\end{proof}

This motivates the definition
\[
	\J_e := \{(k_0,k_1,\dots) : k_i\in\Z_{\geq0}\textnormal{ and }k_i \geq p k_{i+1}\textnormal{ for all }i\geq0\textnormal{ and }k_i = 0\textnormal{ for all }i\geq e\},
\]
so that by \Cref{lem:jump-restrictions}, we have
\[
	(\jump_{P,0}(\varphi),\jump_{P,1}(\varphi),\dots) \in \J_e
	\qquad\textnormal{for all}\qquad\varphi\in\Hom(\Gamma_{K_P},G).
\]
Note that any $k\in\J_e$ satisfies $k_0\geq k_1\geq\dots$ and that all but finitely many $k_i$ must be zero.

\subsection{Higher unit groups}

We next work out a description of the pro-$p$-group $U_P^1$ and its higher-dimensional filtration by products of higher unit groups
\[
	\dots \hookrightarrow U_P^3 \hookrightarrow U_P^2 \hookrightarrow U_P^1,
\]
using \cite[Lemma~4.1]{gundlach-abelian-extensions-by-artin-schreier}. This will later help us count maps $\varphi:\O_P^\times\to G$ with given values for $\jump_{P,i}(\varphi)$.

\begin{definition}
	For any $k\in\J_e$ and any $i\geq1$, let $\lambda_i(k)$ be the smallest integer $\lambda\geq0$ with $i>k_\lambda$, so that
	\[
		k_0 \geq \cdots \geq k_{\lambda-1} \geq i > k_\lambda \geq k_{\lambda+1} \geq \dots.
	\]
	For any $k\in\J_e$ and any place $P$, writing the size of the residue field at $P$ as  $Q_P = q^{\deg(P)} = p^d$, we define the (additive) group
	\[
		R_P^k
		:= \prod_{\substack{i\geq1:\\p\nmid i}} p^{\lambda_i(k)} \Z_p^d
		\subseteq \prod_{\substack{i\geq1:\\p\nmid i}} \Z_p^d.
	\]
\end{definition}

\begin{remark}\label{rmk:properties-of-R}
	We have:
	\begin{enumalpha}
	\item
		$R_P^0 = \prod_{i\geq1:\ p\nmid i}\Z_p^d$ for $k=0=(0,\dots)$.
	\item
		$R_P^k\supseteq R_P^l$ if $k_j\leq l_j$ for all $j\geq0$.
	\end{enumalpha}
\end{remark}

\begin{lemma}\label{lem:image-of-R}
	There is an isomorphism of topological groups $\alpha : R_P^0 \stackrel\sim\to U_P^1$ which for each $k\in\J_e$ sends the subgroup $R_P^k$ of $R_P^0$ to the subgroup of $U_P^1$ jointly generated by the subgroups $(U_P^{k_j+1})^{p^j}$ for $j=0,1,\dots$:
	\[\begin{tikzcd}
		R_P^0 \rar{\sim}[swap]{\alpha} & U_P^1 \\
		R_P^k \uar[hook] \rar{\sim}[swap]{\alpha} & U_P^{k_0+1} \cdot \Big(U_P^{k_1+1}\Big)^p \cdot \Big(U_P^{k_2+1}\Big)^{p^2} \cdots \uar[hook]
	\end{tikzcd}\]
\end{lemma}

\begin{remark}
	If $k_r=k_{r+1}=\dots=0$, then the image of $R_P^k$ is already generated by the finitely many subgroups $(U_P^{k_j+1})^{p^j}$ for $0\leq j<r$ together with the subgroup $(U_P^1)^{p^r}$.
\end{remark}

\begin{proof}
	For any positive real number $x$, let
	\[
		\nu(x) := \min\{k\in\Z_{\geq0} : p^k \geq x\}.
	\]
	For any integer $k\geq0$, let
	\[
		\Lambda_P^{k+1}
		:= \prod_{\substack{i\geq1:\\p\nmid i}} p^{\nu((k+1)/i)}\Z_p^d
		\subseteq \prod_{\substack{i\geq1\\p\nmid i}}\Z_p^d.
	\]
	By \cite[Lemma~4.1]{gundlach-abelian-extensions-by-artin-schreier}, there is an isomorphism $\alpha$ from $\Lambda_P^1 = \prod_{i\geq1:\ p\nmid i} \Z_p^d = R_P^0$ to $U_P^1$ which for each $k\geq0$ sends the subgroup $\Lambda_P^{k+1}$ to the subgroup $U_P^{k+1}$.

	It thus suffices to show that for each $k\in\J_e$, the group $R_P^k$ is the subgroup of $\Lambda_P^1 = R_P^0$ jointly generated by the subgroups $p^j\Lambda_P^{k_j+1}$ for $j=0,1,\dots$. By definition,
	\[
		\Lambda_P^{k_0+1} + p\Lambda_P^{k_1+1} + p^2\Lambda_P^{k_2+1} + \dots
		= \prod_{\substack{i\geq1:\\p\nmid i}}
			p^{\min\{\nu((k_j+1)/i) + j : j\geq0\}} \Z_p^d
		\qquad\textnormal{and}\qquad
		R_P^k
		= \prod_{\substack{i\geq1:\\p\nmid i}}
			p^{\lambda_i(k)} \Z_p^d,
	\]
	so we need to show that for each $i\geq1$ with $p\nmid i$, if we let $\lambda := \lambda_i(k)$, we have
	\begin{equation}\label{eq:nu-eq-lambda}
		\min\{\nu((k_j+1)/i) + j : j\geq0\} = \lambda.
	\end{equation}
	By definition, we have $k_{\lambda-1}\geq i>k_\lambda$.

	We first prove the inequality ``$\geq$'' in \Cref{eq:nu-eq-lambda}. For $j\geq\lambda$, we use $\nu((k_j+1)/i) + j \geq j \geq \lambda$.
	For $0\leq j\leq\lambda-1$, since $k\in\J_e$, we have $k_j \geq p k_{j+1} \geq \dots \geq p^{\lambda-1-j} k_{\lambda-1} \geq p^{\lambda-1-j} i$, so $(k_j+1)/i > p^{\lambda-1-j}$, which implies $\nu((k_j+1)/i) > \lambda-1-j$ and thus $\nu((k_j+1)/i) + j \geq \lambda$.

	To prove the inequality ``$\leq$'', we take $j = \lambda$. Since $k_\lambda<i$, we have $(k_\lambda+1)/i\leq1$, so $\nu((k_\lambda+1)/i) = 0$ and thus $\nu((k_\lambda+1)/i)+\lambda = \lambda$.
\end{proof}

\subsection{Counting local extensions}

Class field theory shows that the sub-$G$-extensions of the local field $K_P$ are in bijection with elements of $\Hom(K^\times_P, G)$. As $K_P^\times \simeq \Z\times\O_P^\times$, we thus have a $|G|$-to-$1$ relation between sub-$G$-extensions of $K_P$ and elements of $\Hom(\O_P^\times, G)$. In this section, our goal is to count those extensions satisfying $\jump_{P,j}(\varphi)\leq k_j$ for all $j\geq0$.

Recall that $\O_P^\times \simeq \kappa_P^\times\times U_P^1 \simeq \kappa_P^\times\times R_P^0$, with the second isomorphism arising from $\alpha$. As $G$ is a $p$-group and the size of $\kappa_P^\times$ is coprime to $p$, we can thus identify $\Hom(\O_P^\times, G)$ with $\Hom(R_P^0, G)$.
For any $k\in\J_e$, we will identify the group $\Hom(R_P^0 / R_P^k, G)$ in the natural way with the group of homomorphisms in $\Hom(R_P^0,G) \simeq \Hom(\O_P^\times, G)$ that vanish on $R_P^k$. \Cref{lem:image-of-R} implies:

\begin{lemma}\label{lem:jump-from-R}
	Let $\varphi\in\Hom(R_P^0,G) = \Hom(\O_P^\times,G)$ and $k\in\J_e$. We have $\jump_{P,j}(\varphi) \leq k_j$ for all $j\geq0$ if and only if $\varphi$ lies in $\Hom(R_P^0/R_P^k, G)$.
\end{lemma}
\begin{proof}
	By \Cref{eq:lastjump-vs-kunits}, we have $\jump_{P,j}(\varphi)\leq k_j$ if and only if $\varphi(U_P^{k_j+1})\subseteq G[p^j]$. This is equivalent to $(U_P^{k_j+1})^{p^j} \subseteq \ker\varphi$. Now, this holds for all $j\geq0$ if and only if the subgroup generated by all $(U_P^{k_j+1})^{p^j}$ is contained in $\ker\varphi$. For the corresponding map $\varphi:R_P^0\to G$, by \Cref{lem:image-of-R}, this means that $R_P^k \subseteq \ker\varphi$.
\end{proof}

Our next goal is to compute the number of homomorphisms $R_P^0/R_P^k \to G$. The formula will involve numbers $c_i,d_i$ defined as follows:

\begin{definition}\label{def:cd}
	To the group $G=\prod_{j=1}^e C_{p^j}^{m_j}$, we associate the following integers $c_i,d_i$, for $i=0,\dots,e$:
	\begin{align*}
		c_i &:= \sum_{j=1}^i m_j (p^i-p^{i-j}) + \sum_{j=i+1}^e m_j p^i, \\
		d_i &:= \sum_{j=1}^i m_j (p^i-p^{i-j}) + \sum_{j=i+1}^e m_j (p^i-1).
	\end{align*}
\end{definition}

\begin{example}
	If $G = C_{p^e}$, then
	\[
		c_i =
		\begin{cases}
			p^i &\textnormal{if }i < e,\\
			p^i - 1 &\textnormal{if }i = e,
		\end{cases}
		\qquad\textnormal{and}\qquad
		d_i = p^i - 1.
	\]
\end{example}

We record some inequalities that will be useful later, when studying the poles of the generating function $F_K(X_0,\dots,X_{e-1})$:

\begin{remark}\label{rmk:cd}
	We have:
	\begin{enumalpha}
	\item\label{rmk:cd:0}
		$c_0 = \sum_{j=1}^e m_j = \dim_{\F_p}(G[p])$ and $d_0 = 0$.
	\item\label{rmk:cd:diff}
		$c_i - d_i = p c_i - d_{i+1} = \sum_{j=i+1}^e m_j > 0$ for $i=0,\dots,e-1$ and $c_e-d_e = 0$.
	\item\label{rmk:cd:diff2}
		$p c_i - c_{i+1} = m_{i+1} \geq 0$ for $i=0,\dots,e-1$.
	\end{enumalpha}
\end{remark}

Now, we are ready to count homomorphisms $R_P^0/R_P^k \to G$:

\begin{lemma}\label{lem:tau}
	For any $k\in\J_e$, if we write $k_i - pk_{i+1} = v_i + p w_i$ with $v_i\in\{0,\dots,p-1\}$ and $w_i\in\Z_{\geq0}$ for $0\leq i\leq e-1$, then
	\[
		|{\Hom(R_P^0/R_P^k, G)}|
		= Q_P^{\sum_{i=0}^{e-1}(c_iv_i+d_{i+1}w_i)},
	\]
	where $Q_P$ is the size of the residue field at $P$.
\end{lemma}
\begin{proof}
	As before, let $Q_P = p^d$.
	By definition,
	\[
		R_P^0 / R_P^k
		= \prod_{\substack{i\geq1:\\p\nmid i}} \left(\Z/p^{\lambda_i(k)}\Z\right)^d,
	\]
	so
	\[
			|{\Hom(R_P^0/R_P^k, G)}|
		= \prod_{\substack{i\geq1:\\p\nmid i}} \big|G\big[p^{\lambda_i(k)}\big]\big|^d.
	\]
	Since $|G[p^r]| = p^{\sum_{j=1}^e m_j\cdot\min(j,r)}$, we obtain
	\[
		|{\Hom(R_P^0/R_P^k, G)}|
		= p^{\tau(k)d} = Q_P^{\tau(k)}
	\]
	with
	\begin{align*}
		\tau(k)
		&:= \sum_{\substack{i\geq1:\\p\nmid i}}
			\sum_{j=1}^e
			m_j \cdot \min(j,\lambda_i(k))
		= \sum_{\substack{i\geq1:\\p\nmid i}}
			\sum_{j=1}^e
			\sum_{u=0}^{\min(j-1,\lambda_i(k)-1)}
			m_j
		= \sum_{j=1}^e m_j
			\sum_{u=0}^{j-1}
			\sum_{\substack{i\geq1:\\p\nmid i\\\lambda_i(k)>u}} 1.
	\end{align*}
	As $\lambda_i(k)>u$ is equivalent to $i\leq k_u$, this means that
	\begin{equation*}
		\tau(k)
		= \sum_{j=1}^e m_j
			\sum_{u=0}^{j-1}
				\bigg(k_u - \left\lfloor\frac{k_u}{p}\right\rfloor\bigg).
	\end{equation*}
	It remains to express this number in terms of the variables $v_i$ and $w_i$.

	For any $k\in\J_e$, we have $k_e=k_{e+1}=\dots=0$ and we have written $k_i - pk_{i+1} = v_i + p w_i$ with $v_i\in\{0,\dots,p-1\}$ and $w_i\in\Z_{\geq0}$ for $0\leq i\leq e-1$. We then have
	\[
		k_u
		= \sum_{i=u}^{e-1} (v_i p^{i-u} + w_i p^{i+1-u})
		= v_u
			+ \sum_{i=u+1}^{e-1} \underbrace{v_i p^{i-u}}_{\in p\Z}
			+ \sum_{i=u}^{e-1} \underbrace{w_i p^{i+1-u}}_{\in p\Z}
		\qquad\textnormal{for any }0\leq u\leq e-1,
	\]
	so
	\[
		\left\lfloor\frac{k_u}{p}\right\rfloor
		= \sum_{i=u+1}^{e-1} v_i p^{i-u-1}
			+ \sum_{i=u}^{e-1} w_i p^{i-u}.
	\]
	Plugging this into our previous formula for $\tau(k)$, we obtain
	\begin{align*}
		\tau(k)
		&= \sum_{j=1}^e m_j \sum_{u=0}^{j-1} \bigg(
			v_u
			+ \sum_{i=u+1}^{e-1} v_i (p^{i-u}-p^{i-u-1})
			+ \sum_{i=u}^{e-1} w_i (p^{i+1-u}-p^{i-u})
		\bigg).
	\end{align*}
	Lastly, we rearrange this expression to write it as a linear combination of the variables $v_i,w_i$:
	\begin{align*}
		\tau(k)
		&= \sum_{j=1}^e m_j \bigg(
			\sum_{i=0}^{j-1} v_i
			+ \sum_{i=0}^{e-1} v_i \sum_{u=0}^{\min(j-1,i-1)} (p^{i-u} - p^{i-u-1})
			+ \sum_{i=0}^{e-1} w_i \sum_{u=0}^{\min(j-1,i)} (p^{i+1-u}-p^{i-u})
		\bigg) \\
		&= \sum_{j=1}^e m_j \bigg(
			\sum_{i=0}^{j-1} v_i
			+ \sum_{i=0}^{e-1} v_i (p^i - p^{i-\min(j,i)})
			+ \sum_{i=0}^{e-1} w_i (p^{i+1} - p^{i+1-\min(j,i+1)})
		\bigg) \\
		&= \sum_{i=0}^{e-1} v_i \bigg(
			\sum_{j=i+1}^e m_j
			+ \sum_{j=1}^e m_j (p^i - p^{i-\min(j,i)})
		\bigg)
		+ \sum_{i=0}^{e-1} w_i \bigg(
			\sum_{j=1}^e m_j (p^{i+1} - p^{i+1-\min(j,i+1)})
		\bigg),
	\end{align*}
	which by the definitions of $c_i$ and $d_i$ can be written as
	\begin{equation*}
		\tau(k) = \sum_{i=0}^{e-1}(c_i v_i + d_{i+1} w_i).
	\qedhere
	\end{equation*}
\end{proof}

\section{Ext filtrations}

In this section, we use what we have learned about $p$-Selmer groups in \Cref{sn:selmer} to gain information about the image under the boundary map $\delta$ of the group $\Hom(R_P^0/R_P^k, G)$ of maps $\varphi$ satisfying $\jump_{P,j}(\varphi)\leq k_j$ for all $j$.

\begin{definition}~
	\begin{enumalpha}
	\item
		For any $P\in M_K$ and $k\in\J_e$, we let $H_P^k\subseteq\Ext^1_\Z(\Pic^0_K, G)$ be the image under the boundary map $\delta$ defined in \Cref{sn:class-field-theory} of the subgroup $\Hom(R_P^0 / R_P^k, G)$ of $\Hom(R_P^0, G) \simeq \Hom(\O_P^\times, G)$.
	\item
		For any $r\geq0$, we let $I^r \subseteq \Ext^1_\Z(\Pic^0_K, G)$ be the image of $\Ext^1_\Z(\Pic^0_K, G[p^r])$ in $\Ext^1_\Z(\Pic^0_K, G)$.
	\end{enumalpha}
\end{definition}

\begin{remark}
	As a consequence of \Cref{rmk:properties-of-R}, we have:
	\begin{enumalpha}
	\item
		$H_P^0 = 1$ for $0=(0,\dots)$.
	\item
		$H_P^k \subseteq H_P^l$ if $k_j\leq l_j$ for all $j\geq0$.
	\end{enumalpha}
\end{remark}

\begin{remark}
	The subgroups $I^r$ satisfy
	\[
		1 = I^0 \subseteq \dots \subseteq I^e = I^{e+1} = \dots = \Ext^1_\Z(\Pic^0_K, G).
	\]
\end{remark}

We will now compare the two filtrations of $\Ext^1_\Z(\Pic_K^0,G)$ by the subgroups $H_P^k$ and $I^r$.

\begin{lemma}\label{lem:filtration-comparison}
	Let $P\in M_K$ and $r\geq0$.
	\begin{enumalpha}
	\item\label{item:filtration-comparison-a}
		If $k,l\in\J_e$ satisfy $k_j=l_j$ for all $j\geq r$, then $H_P^k$ and $H_P^l$ have the same image in the quotient $\Ext^1_\Z(\Pic^0_K,G)/I^r$.
	\item\label{item:filtration-comparison-b}
		Let $w_P$ be as in \Cref{cor:selmer-embedding}\ref{item:selmer-embedding-a}.
		If $k\in\J_e$ satisfies $k_r\geq w_P$, then $I^{r+1} \subseteq H_P^k$.
		In particular, if $k_{e-1} \geq w_P$, then $H_P^k = \Ext^1_\Z(\Pic^0_K, G)$.
	\end{enumalpha}
\end{lemma}
\begin{proof}
	\begin{enumalpha}
	\item
		By definition, we have the natural isomorphisms
		\[
			R_P^0 / R_P^k \simeq \prod_{\substack{i\geq1:\\p\nmid i}} \Z^d/p^{\lambda_i(k)} \Z^d
			\qquad\textnormal{and}\qquad
			R_P^0 / R_P^l \simeq \prod_{\substack{i\geq1:\\p\nmid i}} \Z^d/p^{\lambda_i(l)} \Z^d.
		\]
		We now compare the factor groups.
		Since $k_j=l_j$ for all $j\geq r$, we have $\lambda_i(l) = \lambda_i(k)$ whenever one of the numbers $\lambda_i(l), \lambda_i(k)$ is strictly larger than $r$.
		This means that $R_P^0/R_P^k$ and $R_P^0/R_P^l$ agree except for some $p^r$-torsion factors.
		
		We can add to any $f:R_P^0/R_P^k\to G$ a homomorphism $R_P^0/R_P^k\to G[p^r]$ so that the resulting function factors through $\prod_{i\geq1:\ p\nmid i,\ \lambda_i(k)>r}\Z^d/p^{\lambda_i(k)}\Z^d$.
		Similarly, we can add to any $f:R_P^0/R_P^l\to G$ a homomorphism $R_P^0/R_P^l\to G[p^r]$ so that the resulting function factors through $\prod_{i\geq1:\ p\nmid i,\ \lambda_i(l)>r}\Z^d/p^{\lambda_i(l)}\Z^d$.
		
		Hence, $\Hom(R_P^0/R_P^k,G)$ and $\Hom(R_P^0/R_P^l,G)$ have the same image under $\delta$ in the quotient group $\Ext^1_\Z(\Pic_K^0,G)/\Ext^1_\Z(\Pic_K^0,G[p^r])$, as claimed.
	\item
		Note that for all $1\leq i\leq w_P$, we have $\lambda_i(w_P,0,\dots) = 1$ and $\lambda_i(k)\geq r+1$. For $i>w_P$, we have $\lambda_i(w_P,0,\dots) = 0$. We obtain the following commutative diagram, with the top right arrow being the projection onto the coordinates $i\leq w_P$:
		\[\begin{tikzcd}[column sep=1.5em]
			\O_P^\times/\alpha(R_P^k) \rar \dar[->>] & U_P^1/\alpha(R_P^k) \rar[->>]{\alpha^{-1}} \dar[->>] &
				R_P^0/R_P^k \dar[->>] \rar[->>] &
				\prod_{1\leq i\leq w_P:\ p\nmid i} \Z^d/p^{r+1}\Z^d \dar[->>]\\
			\O_P^\times / (\O_P^{\times p}\cdot U_P^{w_P+1}) \rar{\sim} & U_P^1 / ((U_P^1)^p \cdot U_P^{w_P+1}) \rar{\sim}[swap]{\alpha^{-1}} &
				R_P^0/R_P^{(w_P,0,\dots)} \rar{\sim} &
				\prod_{1\leq i\leq w_P:\ p\nmid i} \Z^d/p\Z^d
		\end{tikzcd}\]
		Recall the explicit description of $\Ext^1_\Z(\Pic^0_K, G)$ and of the boundary map $\delta$ from \Cref{rmk:ext-made-concrete}.

		By \Cref{lem:linear-independence}, the elements $[y_1],\dots,[y_t]$ of the Selmer group $\Sel_K$ are $\F_p$-linearly independent. By \Cref{cor:selmer-embedding}\ref{item:selmer-embedding-a}, so are their images $[y_i/\pi_P^{v_P(y_i)}]$ in the groups in the bottom row of the diagram. Their images can thus be completed to a basis of the bottom right $\F_p$-vector space. Lifting the basis, we see that their images in the top right group can be completed to a $\Z/p^{r+1}\Z$-module basis of the top right group.
		
		In particular, for any elements $g_1,\dots,g_t$ of $G[p^{r+1}]$, there is a map $R_P^0/R_P^k\to G[p^{r+1}]$ (factoring through the top right group) sending $y_i/\pi_P^{v_P(y_i)}$ to $g_i$ for $i=1,\dots,t$. Thus, by the description of $\delta$ given in \Cref{rmk:ext-made-concrete}, every element of $I^{r+1}$ (corresponding to a tuple $([g_1],\dots,[g_t])$ in $\prod_{i=1}^tG[p^{r+1}]/p^{n_i}G[p^{r+1}]$) lies in the image $H_P^k$ of $\Hom(R_P^0/R_P^k, G)$ under the map $\delta$.
	\qedhere
	\end{enumalpha}
\end{proof}

\begin{corollary}\label{cor:filtrations-usually-match}
	For all but finitely many $P\in M_K$, for all $k\in\J_e$, we have
	\[
		H_P^k = I^r
		\qquad\textnormal{if}\qquad
		k_0\geq\dots\geq k_{r-1}>0=k_r=k_{r+1}=\dots.
	\]
\end{corollary}
\begin{proof}
	Applying part (a) of the lemma with $l=0$ and recalling that $H_P^0=1$, we see that $H_P^k \subseteq I^r$. Conversely, according to \Cref{cor:selmer-embedding}, we can take $w_P=1$ for all but finitely many $P$, and then $k_{r-1}\geq w_P$, so part (b) of the lemma implies $I^r \subseteq H_P^k$. (The case $r=0$ is clear as $I^0 = 1$.)
\end{proof}

\begin{corollary}\label{cor:escaping-from-subgroup}
	Let $T$ be a subgroup of $\Ext^1_\Z(\Pic^0_K, G)$.
	Let $r = r(T)\geq0$ be the largest integer $r\leq e$ with $I^r \subseteq T$.
	\begin{enumalpha}
	\item
		For every $P\in M_K$, there is a finite set $S_P=S_P(T)\subseteq\J_{e-r}$ containing $(0,\dots)$ such that for any $k=(k_0,k_1,\dots)\in\J_e$, we have $H_P^k \subseteq T$ if and only if $(k_r,k_{r+1},\dots)\in S_P$.
	\item
		For all but finitely many $P$, we have $S_P = \{(0,\dots)\}$.
	\item
		If $T=\Ext^1_\Z(\Pic_K^0,G)$ or $g_K \leq 1$, then $S_P = \{(0,\dots)\}$ for all $P$.
	\end{enumalpha}
\end{corollary}
\begin{proof}
	The case $T = \Ext^1_\Z(\Pic_K^0,G)$ is clear, so assume that $T$ is a proper subgroup.
	Since $I^r\subseteq T$, the group $H_P^k$ is contained in $T$ if and only if its image in the quotient group $\Ext^1_\Z(\Pic^0_K, G)/I^r$ is contained in $T/I^r$. By \Cref{lem:filtration-comparison}\ref{item:filtration-comparison-a}, this image depends only on $(k_r,k_{r+1},\dots)$ and by \Cref{lem:filtration-comparison}\ref{item:filtration-comparison-b}, a necessary condition is that $k_r<w_P$. (If $k_r\geq w_P$, then $I^{r+1} \subseteq H_P^k \subseteq T$, which would contradict the maximality of $r$.) In particular, there are only finitely many possible values for $k_r$, which together with the condition that $k_i\geq p k_{i+1}$ for all $i$ implies that there are only finitely many possible sequences $(k_r,k_{r+1},\dots)$. This proves (a). If $w_P = 1$, we see that the only possible sequence is $(0,\dots)$, so the other claims about $S_P$ in (b) and (c) then follow from \Cref{cor:selmer-embedding}.
\end{proof}

\begin{remark}\label{rmk:lagemann-higher-exponents}
	The statements in \cite[Section~4]{lagemann-artin-schreier-witt} would imply that $H_P^{(1,0,\dots)} = \Ext^1_\Z(\Pic^0_K, G)$ for all but finitely many $P$, but \Cref{cor:filtrations-usually-match} shows that $H_P^{(1,0,\dots)} = I^1$ for all but finitely many $P$, which is smaller than $\Ext^1_\Z(\Pic^0_K, G)$ if $\Pic^0_K[p] \neq \Pic^0_K[p^2]$ and $G \neq G[p]$ (for example by the concrete description in \Cref{rmk:ext-made-concrete}). The issue occurs in Lagemann's Proposition~4.2. In his notation, for any module $\mathfrak m = \prod_P P^{a_P}$, the kernel of the group homomorphism $S_{i,\mathfrak m} \to S_{1,\mathfrak m}$ is not $S_{1,\mathfrak m}$ but $S_{1,\mathfrak n}$ with $\mathfrak n := \prod_P P^{\lceil a_P/p\rceil}$. Consequently, some of his later results require correction, in particular Lemma~4.5 and Lemma~5.2.
\end{remark}

\section{Generating functions}

\subsection{Character sum}

The exact sequence (\ref{eq:ext-exact-sequence}) allows us to write the generating function $F_K(X_0,\dots,X_{e-1})$ as a sum of finitely many Euler products:

\begin{lemma}\label{lem:character-sum}
	We can write $F_K(X_0,\dots,X_{e-1})$ as the following sum over the characters $\chi$ of the finite group $\Ext^1_\Z(\Pic^0_K, G)$:
	\[
		F_K(X_0,\dots,X_{e-1})
		= \sum_{\chi}
		\prod_{P\in M_K}
		F_{K_P,\chi}\Bigl(X_0^{\deg(P)}, \dots, X_{e-1}^{\deg(P)}\Bigr)
	\]
	with
	\[
		F_{K_P,\chi}(X_0, \dots, X_{e-1})
		:= \sum_{\varphi \in \Hom(R_P^0, G)}
		\chi(\delta(\varphi))
		\prod_{i=0}^{e-1} X^{\jump_{P,i}(\varphi)-p\jump_{P,i+1}(\varphi)}.
	\]
\end{lemma}
Note that the constant term in each Euler factor is~$1$: indeed, for every place $P$, there is exactly one map $\varphi\in\Hom(R_P^0,G)$ with $\jump_{P,i}(\varphi)=0$ for all $i$, namely the trivial map, which has $\chi(\delta(\varphi))=1$.
\begin{proof}
	By the exact sequence (\ref{eq:ext-exact-sequence}) and our identification $\Hom(\O_P^\times, G) = \Hom(R_P^0, G)$, any $\varphi\in\bigoplus_{P\in M_K}\Hom(R_P^0, G)$ has $|G|\cdot|{\Hom(\Pic^0_K, G)}|$ lifts $\A_K^\times / K^\times \to G$ if $\delta(\varphi) = 0$ and no lifts if $\delta(\varphi) \neq 0$.
	By orthogonality of characters, we can thus write $F_K(X_0,\dots,X_{e-1})$ as a sum over all characters $\chi$ of the finite abelian group $\Ext^1_\Z(\Pic^0_K, G)$:
	\begin{align*}
		&F_K(X_0,\dots,X_{e-1}) \\
		&= \frac1{|G|} \sum_{\substack{\varphi\in\Hom(\Gamma_K, G)}}
			\prod_{i=0}^{e-1} X_i^{\deg(\jump_i(\varphi)-p\jump_{i+1}(\varphi))}\\
		&= \frac{|{\Hom(\Pic^0_K,G)}|}{|{\Ext^1_\Z(\Pic^0_K,G)}|}
		\sum_{\chi}
		\sum_{\varphi\in\bigoplus_P\Hom(R_P^0, G)}
		\chi(\delta(\varphi))
		\prod_{i=0}^{e-1} X_i^{\deg(\jump_i(\varphi)-p\jump_{i+1}(\varphi))}.
	\end{align*}

	Since $\Pic^0_K$ is a finite abelian group, we have $|{\Hom(\Pic^0_K, G)}| = |{\Ext^1_\Z(\Pic^0_K,G)}|$. (An isomorphism $\Pic^0_K\simeq\prod_{i=1}^t \Z/m_i\Z$ induces natural isomorphisms $\Hom(\Pic^0_K, G) \simeq \prod_i G[m_i]$ and $\Ext^1_\Z(\Pic^0_K, G) \simeq \prod_i G/m_i G$.)

	Since $\chi\circ\delta$ is a group homomorphism and $\jump_i(\varphi) = \sum_P \jump_{P,i}(\varphi)\cdot P$, this can be written as
	\begin{align*}
		&F_K(X_0,\dots,X_{e-1}) \\
		&= \sum_{\chi}
		\sum_{(\varphi_P)_P\in\bigoplus_P\Hom(R_P^0, G)}
		\prod_{P\in M_K} \left(\chi(\delta(\varphi_P)) \prod_{i=0}^{e-1} X_i^{(\jump_{P,i}(\varphi_P) - p\jump_{P,i+1}(\varphi_P))\cdot\deg(P)}\right) \\
		&= \sum_{\chi} \prod_{P\in M_K} F_{K_P,\chi}\Bigl(X_0^{\deg(P)}, \dots, X_{e-1}^{\deg(P)}\Bigr).
		\qedhere
	\end{align*}
\end{proof}

\subsection{The local generating functions}

In this subsection, we compute the local generating functions $F_{K_P,\chi}(X_0,\dots,X_{e-1})$, generalizing both \cite[Theorem~4.2(b)]{gundlach-abelian-extensions-by-artin-schreier} and \cite[Remark~4.6]{gundlach-abelian-extensions-by-artin-schreier}:

\begin{theorem}\label{thm:local-gf}
	Let $\chi$ be any character of $\Ext^1_\Z(\Pic^0_K, G)$. Let $r := r(\ker\chi)$ be the largest integer $r\leq e$ with $I^r \subseteq\ker\chi$, as in \Cref{cor:escaping-from-subgroup}. Let $c_i,d_i$ be as in \Cref{def:cd}. For $i=0,\dots,e-1$, let
	\[
		H_{P,i}(X_i)
		:= \frac{
			(1 - Q_P^{p c_i} X_i^p) (1 - Q_P^{d_i} X_i)
		}{
			(1 - Q_P^{c_i} X_i) (1 - Q_P^{d_{i+1}} X_i^p)
		}.
	\]
	\begin{enumalpha}
	\item
		If $\chi$ is the trivial character, then
		$
			F_{K_P,\chi}(X_0,\dots,X_{e-1}) = \prod_{i=0}^{e-1} H_{P,i}(X_i)
		$.\\
		Otherwise, $F_{K_P,\chi}(X_0,\dots,X_{e-1}) = (\prod_{i=0}^{r-1} H_{P,i}(X_i))(1-Q_P^{d_{r}}X_{r})B_{P,\chi}(X_{r},\dots,X_{e-1})$
		for some polynomial $B_{P,\chi}(X_{r},\dots,X_{e-1})$ with integer coefficients and constant coefficient~$1$.
	\item
		For all but finitely many $P$, we have $B_{P,\chi}(X_{r},\dots,X_{e-1}) = 1$.
	\item
		If $g_K\leq 1$, then $B_{P,\chi}(X_{r},\dots,X_{e-1}) = 1$ for all $P$.
	\end{enumalpha}
\end{theorem}
\begin{proof}
	Let $r' := \min(e,r+1)$.
	For any $f\in\Z_{\geq0}\cup\{\infty\}$, we define the numbers
	\[
		a^{(f)}(k)
		:= \sum_{\substack{
			\varphi\in\Hom(R_P^0, G):\\
			\jump_{P,j}(\varphi) \leq k_j\textnormal{ for all }0\leq j<f,\\
			\jump_{P,j}(\varphi) = k_j\textnormal{ for all }j\geq f
		}}
			\chi(\delta(\varphi))
		\qquad\textnormal{for }k\in\J_e
	\]
	and the corresponding generating function
	\[
		F^{(f)}(X_0,\dots,X_{e-1})
		:= \sum_{k\in\J_e} a^{(f)}(k) \cdot \prod_{i=0}^{e-1} X_i^{k_i-pk_{i+1}}.
	\]

	Our goal is to compute the function $F_{K_P,\chi}$, which is by definition equal to $F^{(0)}$. We first compute (the coefficients of) the function $F^{(\infty)}$, which turns out to be equal to $F^{(e)}$. This is facilitated by the maps $\varphi$ appearing in the sum defining $a^{(\infty)}$ forming a group. We then use two separate downward inductions: first, we compute $F^{(e-1)},\dots,F^{(r')}$ using that there are only finitely many tuples $(k_r,k_{r+1},\dots)$ for which $a^{(f)}(k)$ can be nonzero; then we compute $F^{(r'-1)},\dots,F^{(0)}$ by using a small trick to turn a straightforward recursion for their coefficients into a recursion for the generating functions.

	\medskip\noindent\textbf{Computing $a^{(e)} = a^{(\infty)}$.}\quad
	Since every $\varphi\in\Hom(R_P^0,G)$ and every $k\in\J_e$ automatically satisfies $\jump_{P,j}(\varphi)=0=k_j$ for all $j\geq e$, we have $a^{(e)}(k) = a^{(\infty)}(k)$ for all $k\in\J_e$.
	
	By \Cref{lem:jump-from-R}, the maps $\varphi$ appearing in the sum defining $a^{(\infty)}(k)$ correspond to the elements of the group $\Hom(R_P^0 / R_P^k, G)$. By orthogonality of characters, the sum vanishes unless $\delta(\Hom(R_P^0/R_P^k,G)) = H^k_P$ is contained in $\ker\chi$. By \Cref{cor:escaping-from-subgroup}, this inclusion is equivalent to $(k_r,k_{r+1},\dots)\in S_{P,\chi}$ for a specific finite subset $S_{P,\chi} = S_P(\ker\chi)$ of $\J_{e-r}$ containing $(0,\dots)$. Thus,
	\[
		a^{(e)}(k) = a^{(\infty)}(k) =
		\begin{cases}
			|{\Hom(R_P^0/R_P^k, G)}| &\textnormal{if }(k_r,k_{r+1},\dots)\in S_{P,\chi},\\
			0 &\textnormal{otherwise}.
		\end{cases}
	\]
	Together with \Cref{lem:tau}, writing $k_i-pk_{i+1}=v_i+pw_i$ with $v_i\in\{0,\dots,p-1\}$ and $w_i\in\Z_{\geq0}$, we see that
	\[
		a^{(e)}(k) = a^{(\infty)}(k) =
		Q_P^{\sum_{i=0}^{r-1}(c_iv_i+d_{i+1}w_i)} \cdot b^{(e)}(k_r,k_{r+1},\dots),
	\]
	where
	\[
		b^{(e)}(k_r,k_{r+1},\dots) :=
		\begin{cases}
			Q_P^{\sum_{i=r}^{e-1}(c_iv_i+d_{i+1}w_i)}
			&\textnormal{if }(k_r,k_{r+1},\dots)\in S_{P,\chi},\\
			0 &\textnormal{otherwise}.
		\end{cases}
	\]
	
	Since $(0,\dots)\in S_{P,\chi}$, we have $b^{(e)}(0,\dots)=1$.
	If $\chi$ is the trivial character, then $r=r'=e$, so $k_r=k_{r+1}=\dots=0$ for all $k\in\J_e$.

	\medskip\noindent\textbf{A recursion.}\quad
	For any $0\leq f < e$, since every $\varphi$ satisfies $\jump_{P,f}(\varphi)\geq p\jump_{P,f+1}(\varphi)$, we have the recursion
	\begin{equation}\label{eq:akf-recursion}
		a^{(f)}(k) =
		\begin{cases}
			a^{(f+1)}(k) - a^{(f+1)}(k_0,\dots,k_{f-1},k_{f}-1,k_{f+1},\dots) &\textnormal{if }k_f > pk_{f+1},\\
			a^{(f+1)}(k) &\textnormal{if }k_f=pk_{f+1}.
		\end{cases}
	\end{equation}
	
	\medskip\noindent\textbf{Computing $a^{(e-1)},\dots,a^{(r')}$.}\quad
	By downward induction, this implies that for every $r'\leq f \leq e$,
	\begin{equation}\label{eq:ar}
		a^{(f)}(k) =
		Q_P^{\sum_{i=0}^{r-1} (c_i v_i + d_{i+1} w_i)} \cdot b^{(f)}(k_r,k_{r+1},\dots)
	\end{equation}
	with
	\[
		b^{(f)}(k_r,k_{r+1},\dots) :=
		\begin{cases}
			b^{(f+1)}(k_r,\dots) - b^{(f+1)}(k_r,\dots,k_{f-1},k_{f}-1,k_{f+1},\dots) &\textnormal{if }k_f > pk_{f+1},\\
			b^{(f+1)}(k_r,\dots) &\textnormal{if }k_f=pk_{f+1}.
		\end{cases}
	\]
	
	We have $b^{(f)}(0,\dots)=1$.
	If $\chi$ is nontrivial, then $r' = r+1$. As the set $S_{P,\chi}$ is finite, there are only finitely many different values $k_r$ for which we can have $b^{(f)}(k_r,k_{r+1},\dots)\neq0$. The inequalities $k_r\geq k_{r+1}\geq\dots$ on $\J_e$ then imply that each $b^{(f)}$ has finite support. By \Cref{cor:escaping-from-subgroup}, we have $S_{P,\chi}=\{(0,\dots)\}$, for all but finitely many $P$, and if $g_K\leq1$ holds, then for all $P$. In this case, the only possible value for $k_r$ is $0$, so the support of $b^{(f)}$ is also $\{(0,\dots)\}$.

	\medskip\noindent\textbf{Computing $F^{(r')}$.}\quad
	We define
	\[
		B(X_r,\dots,X_{e-1})
		:= \sum_{\substack{
				(k_r,k_{r+1},\dots)\in\J_{e-r}
			}}
			b^{(r')}(k_r,k_{r+1},\dots)
			\cdot \prod_{i=r}^{e-1} X_i^{k_i - pk_{i+1}}.
	\]
	We have shown above that $B$ is a polynomial with constant coefficient $1$, and that it is the constant polynomial $1$ for all but finitely many $P$ and if $g_K\leq1$ holds, then for all $P$.

	Now,
	\begin{align*}
		F^{(r')}(X_0,\dots,X_{e-1})
		&= \sum_{\substack{
			k\in\J_e
		}}
			a^{(r')}(k)
			\cdot \prod_{i=0}^{e-1} X_i^{k_i-pk_{i+1}} \\
		&\stackrel{\mathclap{(\ref{eq:ar})}}= \sum_{\substack{
			k\in\J_e:\\
			k_i-pk_{i+1}=v_i+pw_i
		}}
			Q_P^{\sum_{i=0}^{r-1}(c_i v_i + d_{i+1}w_i)}
			\cdot b^{(r')}(k_r,k_{r+1},\dots)
			\cdot \prod_{i=0}^{e-1} X_i^{k_i-pk_{i+1}} \\
		&= \sum_{\substack{
			v_0,\dots,v_{r-1}\in\{0,\dots,p-1\}\\
			w_0,\dots,w_{r-1}\in\Z_{\geq0}
		}}
			Q_P^{\sum_{i=0}^{r-1} (c_i v_i + d_{i+1} w_i)}
			\cdot B(X_r,\dots,X_{e-1})
			\cdot \prod_{i=0}^{r-1} X_i^{v_i+pw_i} \\
		&= \prod_{i=0}^{r-1}
			\left(
				\bigg(\sum_{v_i=0}^{p-1} Q_P^{c_i v_i} X_i^{v_i}\bigg)
				\bigg(\sum_{w_i=0}^\infty Q_P^{d_{i+1}w_i} X_i^{pw_i}\bigg)
			\right)
			\cdot B(X_r,\dots,X_{e-1}) \\
		&= \prod_{i=0}^{r-1}
			\left(
				\frac{1-(Q_P^{c_i}X_i)^p}{1-Q_P^{c_i}X_i}
				\cdot \frac{1}{1-Q_P^{d_{i+1}}X_i^p}
			\right)
			\cdot B(X_r,\dots,X_{e-1}).
	\end{align*}

	\medskip\noindent\textbf{Computing $F^{(r'-1)},\dots,F^{(0)}$.}\quad
	The recursion for the coefficients $a^{(f)}(k)$ does not \emph{immediately} produce a recursion for the generating functions $F^{(f)}$ because $k'\in\J_e$ does not imply $(k_0',\dots,k_{f-1}',k_f'+1,k_{f+1}',\dots)\in\J_e$. We need a small trick: we first use downward induction over $h$ to prove that for all $0\leq g < h\leq r'$ and all $k\in\J_e$, we have
	\begin{equation}\label{eq:akadd}
		a^{(h)}(k_0+p^g,\dots,k_{g-1}+p,k_g,k_{g+1},\dots)
		= Q_P^{d_g} \cdot a^{(h)}(k).
	\end{equation}

	For $h=r'$, this formula follows from \Cref{eq:ar} since $g\leq r$. (For $g\geq1$, the change of variables $k\mapsto(k_0+p^g,\dots,k_{g-1}+p,k_g,k_{g+1})$ increases the value of $w_{g-1}$ in \Cref{eq:ar} by~$1$ and does not change any other values $v_i,w_i$. For $g=0$, it does not change any of the values, and we have $d_0=0$.)

	The induction step follows from our recursion formula in \Cref{eq:akf-recursion} because the change of variables $k\mapsto(k_0+p^g,\dots,k_{g-1}+p,k_g,k_{g+1})$ commutes with $k\mapsto(k_0,\dots,k_{f-1},k_f-1,k_{f+1},\dots)$.

	By \Cref{eq:akf-recursion}, for any $0\leq f < r'$, we have
	\begin{align*}
		& F^{(f+1)}(X_0,\dots,X_{e-1}) - F^{(f)}(X_0,\dots,X_{e-1}) \\
		&= \sum_{\substack{k\in\J_e:\\k_f>pk_{f+1}}}
			a^{(f+1)}(k_0,\dots,k_{f-1},k_f-1,k_{f+1},\dots)
			\cdot \prod_{i=0}^{e-1} X_i^{k_i-pk_{i+1}}.
	\end{align*}
	We now perform the change of variables
	\[
		k \mapsto
		k' := (k_0-p^f,\dots,k_{f-1}-p,k_{f}-1,k_{f+1},\dots).
	\]
	The important point here is that $k'$ lies in $\J_e$ if and only if $k$ lies in $\J_e$ and satisfies $k_f>pk_{f+1}$. Applying \Cref{eq:akadd}, we thus obtain
	\begin{align*}
		&F^{(f+1)}(X_0,\dots,X_{e-1}) - F^{(f)}(X_0,\dots,X_{e-1}) \\
		&= \sum_{k'\in\J_e}
			Q_P^{d_f} a^{(f+1)}(k') \cdot X_f \prod_{i=0}^{e-1} X_i^{k_i'-pk_{i+1}'}
		= Q_P^{d_f} X_f \cdot F^{(f+1)}(X_0,\dots,X_{e-1}),
	\end{align*}
	so
	\[
		F^{(f)}(X_0,\dots,X_{e-1})
		= F^{(f+1)}(X_0,\dots,X_{e-1}) \cdot (1 - Q_P^{d_f} X_f)
		\qquad\textnormal{for all }0\leq f < r'.
	\]
	Therefore,
	\begin{equation}\label{eq:FKPtriv-from-Fe}
	\begin{aligned}
		F_{K_P,\chi}(X_0,\dots,X_{e-1})
		&= F^{(0)}(X_0,\dots,X_{e-1}) 
		= F^{(r')}(X_0,\dots,X_{e-1}) \cdot \prod_{i=0}^{r'-1} (1-Q_P^{d_i}X_i).
	\end{aligned}
	\end{equation}
	If $\chi$ is the trivial character, then $r=r'=e$ and $B=1$. Otherwise, $r'=r+1$. The result follows by plugging in our formula for $F^{(r')}$.
\end{proof}

\subsection{The global generating function}

We can now describe the global generating function $F_K(X_0,\dots,X_{e-1})$ in terms of the Hasse--Weil zeta function of $K$,
\[
	Z_K(X)
	= \sum_{D\geq0\textnormal{ divisor on }C_K} X^{\deg(D)}
	= \prod_{P\in M_K} \frac{1}{1-X^{\deg(P)}},
\]
which is well known to be a rational function.

For $i=0,\dots,e-1$, we define
\[
	H_i(X_i)
	:= \frac{
		Z_K(q^{c_i} X_i) Z_K(q^{d_{i+1}} X_i^p)
	}{
		Z_K(q^{p c_i} X_i^p) Z_K(q^{d_i} X_i)
	}.
\]
Recalling the definition of $H_{P,i}(X_i)$ in \Cref{thm:local-gf} and that $Q_P = q^{\deg(P)}$, we see that
\[
	\prod_{P\in M_K} H_{P,i}{\left(X_i^{\deg(P)}\right)} = H_i(X_i).
\]

\begin{theorem}[cf.\ \Cref{thm:intro-main}]\label{thm:main}
	The generating function $F_K(X_0,\dots,X_{e-1})$ is rational. Specifically, the generating function is of the form
	\begin{equation*}
		F_K(X_0,\dots,X_{e-1})
		= \prod_{i=0}^{e-1} H_i(X_i)
		+ \sum_{\chi\neq\triv}
			\left(\prod_{i=0}^{r_\chi-1}
			H_i(X_i)\right)
			\cdot \frac1{Z_K(q^{d_{r_\chi}} X_{r_\chi})}
			\cdot B_{\chi}(X_{r_\chi},\dots,X_{e-1}),
	\end{equation*}
	where the sum is over all nontrivial characters $\chi$ of the finite group $\Ext^1_\Z(\Pic_K^0, G)$ and $c_0,\dots,c_e$ and $d_0,\dots,d_e$ are as in \Cref{def:cd}.

	For any nontrivial character, $0\leq r_\chi<e$ and $B_\chi$ is a polynomial with integer coefficients and constant coefficient $1$. If $K$ has genus $g_K\leq1$, then $B_\chi=1$.
\end{theorem}
\begin{proof}
	Let $r_\chi := r(\ker\chi)$.
	By \Cref{lem:character-sum} and \Cref{thm:local-gf}, we obtain
	\begin{align*}
		&F_K(X_0,\dots,X_{e-1}) \\
		&= \sum_{\chi}
			\prod_{P\in M_K}
			F_{K_P,\chi}\Bigl(X_0^{\deg(P)}, \dots, X_{e-1}^{\deg(P)}\Bigr) \\
		&= \prod_{P\in M_K}
			\prod_{i=0}^{e-1}
				H_{P,i}(X_i^{\deg(P)}) \\
		&\quad+ \sum_{\chi\neq\triv}
			\prod_{P\in M_K}
				\left(\prod_{i=0}^{r_\chi-1}
				H_{P,i}\left(X_i^{\deg(P)}\right)\right)
				\cdot (1 - (q^{d_{r_\chi}} X_{r_\chi})^{\deg(P)})
				\cdot B_{P,\chi}(X_{r_\chi}^{\deg(P)},\dots,X_{e-1}^{\deg(P)}) \\
		&= \prod_{i=0}^{e-1}
			H_i(X_i)
		+ \sum_{\chi\neq\triv}
			\left(\prod_{i=0}^{r_\chi-1}
			H_i(X_i)\right)
			\cdot \frac1{Z_K(q^{d_{r_\chi}} X_{r_\chi})}
			\cdot \prod_{P\in M_K} B_{P,\chi}(X_{r_\chi}^{\deg(P)},\dots,X_{e-1}^{\deg(P)}).
	\end{align*}
	Here, $B_\chi(X_r,\dots,X_{e-1}) := \prod_{P\in M_K} B_{P,\chi}(X_r^{\deg(P)},\dots,X_{e-1}^{\deg(P)})$ is always a polynomial and $B_\chi = 1$ if $g_K\leq1$.
\end{proof}

\begin{example}
	Let $G = C_p$, so $e=1$, and let $K$ be the function field of an elliptic curve $E$ over $\F_q$. Then, $\Pic_K^0 \simeq E(\F_q)$. Say $E(\F_q)[p] \simeq C_p^a$. According to \cite[\href{https://stacks.math.columbia.edu/tag/0C1Z}{Lemma 0C1Z~(2)}]{stacks-project}, we always have $a\in\{0,1\}$. By \Cref{rmk:ext-made-concrete}, we then have $\Ext^1_\Z(\Pic_K^0, G) \simeq C_p^a$. It follows that the generating function is
	\[
		F_K(X) = \frac{Z_K(qX) Z_K(q^{p-1} X^p)}{Z_K(q^p X^p) Z_K(X)} + (p^a-1)\cdot\frac{1}{Z_K(X)}.
	\]
\end{example}

\section{Poles and coefficient asymptotics}

We now determine the innermost poles of $F_K(X_0,\dots,X_{e-1})$ and deduce asymptotic statements on the number of sub-$G$-extensions with prescribed heights.

First, consider the index $i$ factor
	\[
		H_i(X_i)
		= \frac{
			Z_K(q^{c_i} X_i) Z_K(q^{d_{i+1}} X_i^p)
		}{
			Z_K(q^{p c_i} X_i^p) Z_K(q^{d_i} X_i)
		}.
	\]
appearing in \Cref{thm:main}. We show that its innermost pole is a simple pole at $X_i = q^{-1-c_i}$:

\begin{lemma}\label{lem:factor-cont}
	Let $0\leq i\leq e-1$.
	\begin{enumalpha}
	\item\label{lem:factor-cont:a}
		The rational function $\tilde H_i(X_i) := H_i(X_i)\cdot(1-q^{1+c_i}X_i)$ is holomorphic in (an open neighborhood of) the disc $\{|X_i|\leq q^{-1-c_i}\}$ and satisfies
		\[
			C_i
			:= \tilde H_i(q^{-1-c_i})
			= \frac{
				(-q\Res_{X=q^{-1}}Z_K(X))
				Z_K(q^{-p+d_{i+1}-pc_i})
			}{
				Z_K(q^{-p})
				Z_K(q^{-1+d_i-c_i})
			}
			> 0.
		\]
	\item\label{lem:factor-cont:b}
		The $X_i^{n_i}$-coefficient of $H_i(X_i) =: \sum_{n_i\geq0} N_i(n_i) X_i^{n_i}$ satisfies
		\[
			N_i(n_i) \sim C_i q^{(1+c_i)n_i} \qquad\textnormal{for }n_i\to\infty.
		\]
	\end{enumalpha}
\end{lemma}
\begin{proof}
	\begin{enumalpha}
	\item
		In the disc $\{|X|\leq q^{-1}\}$, the Hasse--Weil zeta function $Z_K(X)$ has a simple pole at $X=q^{-1}$ with negative residue and no other poles or zeros. The claim follows since we have $c_i > d_i$ and $pc_i \geq c_{i+1}\geq d_{i+1}$ by \Cref{rmk:cd}\ref{rmk:cd:diff}\ref{rmk:cd:diff2}.
	\item
		The idea is to subtract the appropriate geometric series from $H_i(X_i)$ to eliminate the pole at $X_i=q^{-1-c_i}$ and then use the convergence criterion for power series. (See for example \cite[Theorem~IV.9]{flajolet-sedgewick-analytic-combinatorics}.)
		The rational function $H_i(X_i) - C_i / (1-q^{1+c_i}X_i)$ is holomorphic on an open neighborhood of the disc $\{|X_i|\leq q^{-1-c_i}\}$. Considering this function as a power series in $X_i$, we conclude that its radius of convergence is strictly larger than $q^{-1-c_i}$, so its $X_i^{n_i}$-coefficient is $o(q^{(1+c_i)n_i})$. The $X_i^{n_i}$-coefficient of the geometric series $C_i/(1-q^{1+c_i}X_i)$ we subtracted is $C_iq^{(1+c_i)n_i}$.
	\qedhere
	\end{enumalpha}
\end{proof}

Next, we consider the factors $1/Z_K(q^{d_r}X_r)$ appearing in \Cref{thm:main}:

\begin{lemma}\label{lem:extra-cont}
	Let $0\leq i<e$.
	\begin{enumalpha}
	\item\label{lem:extra-cont:a}
		The rational function $1/Z_K(q^{d_i}X_i)$ is holomorphic in (an open neighborhood of) the disc $\{|X_i|\leq q^{-1-c_i}\}$.
	\item\label{lem:extra-cont:b}
		The $X_i^{n_i}$-coefficient of $1/Z_K(q^{d_i}X_i) =: \sum_{n_i\geq0} M_i(n_i) X_i^{n_i}$ satisfies
		\[
			M_i(n_i) = o(q^{(1+c_i)n_i}) \qquad\textnormal{for }n_i\to\infty.
		\]
	\end{enumalpha}
\end{lemma}
\begin{proof}
	(a) follows from $c_i>d_i$ as the Hasse--Weil zeta function $Z_K(X)$ has no zeros in the disc $\{|X|\leq q^{-1}\}$. Then, (b) follows from the convergence criterion for power series.
\end{proof}

\begin{theorem}\label{thm:Fasymp}
	Let $C_0,\dots,C_{e-1}$ be as in \Cref{lem:factor-cont}\ref{lem:factor-cont:a}.
	
	The $X_0^{n_0}\cdots X_{e-1}^{n_{e-1}}$-coefficient $N(n_0,\dots,n_{e-1})$ of $F_K(X_0,\dots,X_{e-1})$, which is by definition $1/|G|$ times the number of $\varphi\in\Hom(\Gamma_K,G)$ with $\deg(\jump_i(\varphi)-p\cdot\jump_{i+1}(\varphi))=n_i$ for all $i=0,\dots,e-1$, satisfies
	\begin{equation}\label{eq:Fasymp:a}
		N(n_0,\dots,n_{e-1})
		\sim C_0\cdots C_{e-1} \prod_{i=0}^{e-1} q^{(1+c_i) n_i}
		\qquad\textnormal{for }n_0,\dots,n_{e-1}\to\infty
	\end{equation}
	and
	\begin{equation}\label{eq:Fasymp:b}
		N(n_0,\dots,n_{e-1}) \ll \prod_{i=0}^{e-1} q^{(1+c_i)n_i}
		\qquad\textnormal{for all }n_0,\dots,n_{e-1}\geq0.
	\end{equation}
\end{theorem}
\begin{proof}
	Recall the statements about $N_i(n_i)$ and $M_r(n_r)$ in \Cref{lem:factor-cont}\ref{lem:factor-cont:b} and \Cref{lem:extra-cont}\ref{lem:extra-cont:b}. Consider the formula for the generating function given in \Cref{thm:main}.

	The contribution to $N(n_0,\dots,n_{e-1})$ arising from the first summand (corresponding to the trivial character) is
	\[
		N_0(n_0)\cdots N_{e-1}(n_{e-1}),
	\]
	which is
	\[
		\sim C_0\cdots C_{e-1}\prod_{i=0}^{e-1} q^{(1+c_i) n_i}
		\qquad\textnormal{for }n_0,\dots,n_{e-1}\to\infty
	\]
	and
	\[
		\ll \prod_{i=0}^{e-1} q^{(1+c_i) n_i}
		\qquad\textnormal{for all }n_0,\dots,n_{e-1}\geq0.
	\]

	The contribution from any nontrivial character $\chi$ is
	\[
		\ll \sum_{(t_{r_\chi},\dots,t_{e-1})\in\operatorname{supp}(B_\chi)}
		N_0(n_0)\cdots N_{r_\chi-1}(n_{r_\chi-1})
		\cdot M_{r_\chi}(n_{r_\chi}-t_{r_\chi}),
	\]
	which is
	\[
		o{\left(\prod_{i=0}^{e-1} q^{(1+c_i) n_i}\right)}
		\qquad\textnormal{for }n_0,\dots,n_{e-1}\to\infty
	\]
	and
	\[
		\ll \prod_{i=0}^{e-1} q^{(1+c_i) n_i}
		\qquad\textnormal{for all }n_0,\dots,n_{e-1}\geq0.
	\qedhere
	\]
\end{proof}

\begin{corollary}\label{cor:Fcont}
	Consider the rational function
	\[
		\tilde F_K(X_0,\dots,X_{e-1})
		:= F_K(X_0,\dots,X_{e-1})\cdot\prod_{i=0}^{e-1}(1-q^{1+c_i}X_i).
	\]
	\begin{enumalpha}
	\item\label{cor:Fcont:a}
		This function is holomorphic in an open neighborhood of the polydisc $\{|X_i|\leq q^{-1-c_i}:i=0,\dots,e-1\}$ and satisfies $\tilde F_K(q^{-1-c_0},\dots,q^{-1-c_{e-1}}) = \prod_{i=0}^{e-1} C_i$ with $C_0,\dots,C_{e-1}$ as in \Cref{lem:factor-cont}\ref{lem:factor-cont:a}.
	\item\label{cor:Fcont:b}
		We have $\tilde F_K(X_0,\dots,X_{e-1})>0$ for any real numbers $0<X_i\leq q^{-1-c_i}$.
	\end{enumalpha}
\end{corollary}
\begin{proof}
	\begin{enumalpha}
	\item
		This follows immediately from \Cref{thm:main}, \Cref{lem:factor-cont}\ref{lem:factor-cont:a}, and \Cref{lem:extra-cont}\ref{lem:extra-cont:a}. (Only the trivial character contributes to the value $\tilde F_K(q^{-1-c_0},\dots,q^{-1-c_{e-1}})$.)
	\item
		The upper bound in \Cref{eq:Fasymp:b} implies that the power series
		\[
			F_K(X_0,\dots,X_{e-1}) = \sum_{n_0,\dots,n_{e-1}\geq0} N(n_0,\dots,n_{e-1}) X_0^{n_0}\cdots X_{e-1}^{n_{e-1}}
		\]
		is absolutely convergent in the polydisc $\{|X_i|<q^{-1-c_i}:i=0,\dots,e-1\}$. The asymptotic statement in \Cref{eq:Fasymp:a} implies that there is an integer $M\geq0$ such that for all $n_0,\dots,n_{e-1}\geq M$, we have $N(n_0,\dots,n_{e-1}) \geq \frac12 C_0\cdots C_{e-1} \prod_{i=0}^{e-1} q^{(1+c_i)n_i}$. But then,
		\begin{align*}
			F_K(X_0,\dots,X_{e-1})
			&\geq \sum_{n_0,\dots,n_{e-1}\geq M} \frac12 C_0\cdots C_{e-1} \prod_{i=0}^{e-1} (q^{1+c_i}X_i)^{n_i}
			= \frac12 C_0\cdots C_{e-1} \prod_{i=0}^{e-1} \frac{(q^{1+c_i}X_i)^M}{1 - q^{1+c_i}X_i}
		\end{align*}
		for any real numbers $0<X_i<q^{-1-c_i}$.
		Thus,
		\[
			\tilde F_K(X_0,\dots,X_{e-1}) \geq \frac12 C_0\cdots C_{e-1} \prod_{i=0}^{e-1} (q^{1+c_i}X_i)^M > 0
		\]
		for all real numbers $0<X_i<q^{-1-c_i}$, and this extends by continuity to $0<X_i\leq q^{-1-c_i}$.
	\qedhere
	\end{enumalpha}
\end{proof}

\begin{theorem}[cf.\ \Cref{thm:intro-main-asc}]\label{thm:main-asc}
	Let $a := 1 + \dim_{\F_p}(G[p])$.
	\begin{enumalpha}
	\item
		The generating function
		\[
			\Fasc_K(X)
			:= \frac1{|G|}\sum_{\substack{\varphi\in\Hom(\Gamma_K, G)}} X^{\deg(\asc(\varphi))}
			= F_K(X,X^p,\dots,X^{p^{e-1}})
			\in \Z\llbracket X\rrbracket.
		\]
		is rational.
	\item
		Its only innermost pole is a simple pole at $X = q^{-a}$ with negative residue.
	\item
		We have the following asymptotic statement for some constant $C > 0$:
		\[
			|\{\varphi\in\Hom(\Gamma_K, G) : \deg(\asc(\varphi)) = n\}|
			\sim C q^{a n}
			\qquad\textnormal{ for }n\rightarrow\infty.
		\]
	\end{enumalpha}
\end{theorem}
\begin{proof}
	\begin{enumalpha}
	\item
		Since
		\[
			\asc(\varphi) = \jump_0(\varphi) = \sum_{i=0}^{e-1} (\jump_i(\varphi) - p\jump_{i+1}(\varphi))\cdot p^i,
		\]
		we have
		\begin{align*}
			\Fasc_K(X)
			&= \frac{1}{|G|} \sum_{\varphi\in\Hom(\Gamma_K,G)} X^{\deg(\asc(\varphi))} \\
			&= \frac{1}{|G|} \sum_{\varphi\in\Hom(\Gamma_K,G)} X^{\sum_{i=0}^{e-1}\deg(\jump_i(\varphi)-p\jump_{i+1}(\varphi))\cdot p^i} \\
			&= F_K(X,X^p,\dots,X^{p^{e-1}}).
		\end{align*}
		 Hence, the rationality of the single-variable generating function $\Fasc_K(X)$ follows immediately from the rationality of the multivariate generating function $F_K(X_0,\dots,X_{e-1})$.
	\item
		This follows from \Cref{cor:Fcont} and the equality
		\[
			\Fasc_K(X)
			= \frac{\tilde F_K(X,X^p,\dots,X^{p^{e-1}})}{\prod_{i=0}^{e-1}\left(1-q^{1+c_i}X^{p^i}\right)}.
		\]
		The pole comes from the factor corresponding to $i=0$. The poles arising from the factors with $i>0$ have a larger modulus as by \Cref{rmk:cd}\ref{rmk:cd:0}\ref{rmk:cd:diff2}, we have
		\[
			\dim_{\F_p}(G[p]) = c_0 \geq c_1/p \geq c_2/p^2 \geq \dots > 0
		\]
		and thus $|X|\leq q^{-a}$ implies $|q^{1+c_i}X^{p^i}| \leq q^{1+c_i-ap^i} = q^{1-p^i+c_i-c_0 p^i} \leq q^{1-p^i} < 1$.

		The value of the numerator at $X=q^{-a}$ is positive by \Cref{cor:Fcont}\ref{cor:Fcont:b}, so the residue of $\Fasc_K(X)$ at this point is negative.
	\item
		The asymptotic statement on the coefficients of the power series $\Fasc_K(X)$ again follows by subtracting the appropriate geometric series from $\Fasc_K(X)$ to remove the innermost pole. (See for example \cite[Theorem~IV.9]{flajolet-sedgewick-analytic-combinatorics}.)
	\qedhere
	\end{enumalpha}
\end{proof}

\bibliographystyle{alphaurl}
\bibliography{references}

\end{document}